\documentclass[12pt]{amsart}
\usepackage{float} 
\usepackage{amsmath,amsthm,amssymb,enumitem,verbatim,stmaryrd,xcolor,microtype,graphicx,aliascnt,needspace}
\usepackage[T1]{fontenc}
\usepackage[utf8]{inputenc}
\usepackage[english]{babel} 
\usepackage[top=2.95cm,bottom=2.95cm,left=3.5cm,right=3.5cm]{geometry}
\usepackage[bookmarksdepth=2,linktoc=page,colorlinks,linkcolor={red!70!black},citecolor={red!80!black},urlcolor={blue!80!black}]{hyperref}

\usepackage{tikz}
\usepackage{tikz-cd}\usetikzlibrary{arrows}

\usepackage[mathcal]{euscript}                               

\usepackage{mathptmx}                                        
\usepackage[colorinlistoftodos,bordercolor=orange,backgroundcolor=orange!20,linecolor=orange,textsize=footnotesize]{todonotes} \makeatletter \providecommand \@dotsep{5} \def\listtodoname{List of Todos} \def\listoftodos{\@starttoc{tdo}\listtodoname} \makeatother 

\allowdisplaybreaks 

\usepackage{etoolbox}
\makeatletter
\patchcmd{\@startsection}{\@afterindenttrue}{\@afterindentfalse}{}{}             
\patchcmd{\part}{\bfseries}{\bfseries\LARGE}{}{}
\patchcmd{\section}{\scshape}{\bfseries}{}{}\renewcommand{\@secnumfont}{\bfseries} 
\patchcmd{\@settitle}{\uppercasenonmath\@title}{\large}{}{}
\patchcmd{\@setauthors}{\MakeUppercase}{}{}{}
\addto{\captionsenglish}{} 
\addto{\captionsenglish}{} 
\addto{\captionsenglish}{} 
\makeatother

\usepackage{fancyhdr}

\addto\extrasenglish{}  
\theoremstyle{plain}
\newtheorem{thm}{Theorem}[section] 
\newaliascnt{lemma}{thm}\aliascntresetthe{lemma}
\newaliascnt{cor}{thm}\newtheorem{cor}[cor]{Corollary}\aliascntresetthe{cor}
\newaliascnt{prop}{thm}\newtheorem{prop}[prop]{Proposition}\aliascntresetthe{prop}

\newtheorem*{thm*}{Theorem}
\newtheorem*{lem*}{Lemma}
\newtheorem*{cor*}{Corollary}

\theoremstyle{definition}
\newaliascnt{df}{thm}\newtheorem{df}[df]{Definition}\aliascntresetthe{df}
\newaliascnt{rem}{thm}\newtheorem{rem}[rem]{Remark}\aliascntresetthe{rem}
\newaliascnt{ex}{thm}\newtheorem{ex}[ex]{Example}\aliascntresetthe{ex}

\newtheorem*{df*}{Definition}
\newtheorem*{ex*}{Example}
\newtheorem*{rem*}{Remark}

\theoremstyle{remark}

\DeclareRobustCommand{\gobblefour}[5]{}    

\DeclareFontFamily{OT1}{pzc}{}                                
\DeclareFontShape{OT1}{pzc}{m}{it}{<-> s * [1.10] pzcmi7t}{}
\DeclareMathAlphabet{\mathpzc}{OT1}{pzc}{m}{it}
\DeclareSymbolFont{sfoperators}{OT1}{bch}{m}{n} \DeclareSymbolFontAlphabet{\mathsf}{sfoperators} \makeatletter\def\operator@font{\mathgroup\symsfoperators}\makeatother 
\DeclareSymbolFont{cmletters}{OML}{cmm}{m}{it}              
\DeclareSymbolFont{cmsymbols}{OMS}{cmsy}{m}{n}
\DeclareSymbolFont{cmlargesymbols}{OMX}{cmex}{m}{n}
\DeclareMathSymbol{\myjmath}{\mathord}{cmletters}{"7C}     \let\jmath\myjmath 
\DeclareMathSymbol{\myamalg}{\mathbin}{cmsymbols}{"71}     
\DeclareMathSymbol{\mycoprod}{\mathop}{cmlargesymbols}{"60}
\DeclareMathSymbol{\myalpha}{\mathord}{cmletters}{"0B}     \let\alpha\myalpha 
\DeclareMathSymbol{\mybeta}{\mathord}{cmletters}{"0C}      \let\beta\mybeta
\DeclareMathSymbol{\mygamma}{\mathord}{cmletters}{"0D}     \let\gamma\mygamma
\DeclareMathSymbol{\mydelta}{\mathord}{cmletters}{"0E}     \let\delta\mydelta
\DeclareMathSymbol{\myepsilon}{\mathord}{cmletters}{"0F}   \let\epsilon\myepsilon
\DeclareMathSymbol{\myzeta}{\mathord}{cmletters}{"10}      \let\zeta\myzeta
\DeclareMathSymbol{\myeta}{\mathord}{cmletters}{"11}       \let\eta\myeta
\DeclareMathSymbol{\mytheta}{\mathord}{cmletters}{"12}     \let\theta\mytheta
\DeclareMathSymbol{\myiota}{\mathord}{cmletters}{"13}      \let\iota\myiota
\DeclareMathSymbol{\mykappa}{\mathord}{cmletters}{"14}     \let\kappa\mykappa
\DeclareMathSymbol{\mylambda}{\mathord}{cmletters}{"15}    \let\lambda\mylambda
\DeclareMathSymbol{\mymu}{\mathord}{cmletters}{"16}        \let\mu\mymu
\DeclareMathSymbol{\mynu}{\mathord}{cmletters}{"17}        \let\nu\mynu
\DeclareMathSymbol{\myxi}{\mathord}{cmletters}{"18}        \let\xi\myxi
\DeclareMathSymbol{\mypi}{\mathord}{cmletters}{"19}        \let\pi\mypi
\DeclareMathSymbol{\myrho}{\mathord}{cmletters}{"1A}       \let\rho\myrho
\DeclareMathSymbol{\mysigma}{\mathord}{cmletters}{"1B}     \let\sigma\mysigma
\DeclareMathSymbol{\mytau}{\mathord}{cmletters}{"1C}       \let\tau\mytau
\DeclareMathSymbol{\myupsilon}{\mathord}{cmletters}{"1D}   \let\upsilon\myupsilon
\DeclareMathSymbol{\myphi}{\mathord}{cmletters}{"1E}       \let\phi\myphi
\DeclareMathSymbol{\mychi}{\mathord}{cmletters}{"1F}       \let\chi\mychi
\DeclareMathSymbol{\mypsi}{\mathord}{cmletters}{"20}       \let\psi\mypsi
\DeclareMathSymbol{\myomega}{\mathord}{cmletters}{"21}     \let\omega\myomega
\DeclareMathSymbol{\myvarepsilon}{\mathord}{cmletters}{"22}\let\varepsilon\myvarepsilon
\DeclareMathSymbol{\myvartheta}{\mathord}{cmletters}{"23}  \let\vartheta\myvartheta
\DeclareMathSymbol{\myvarpi}{\mathord}{cmletters}{"24}     \let\varpi\myvarpi
\DeclareMathSymbol{\myvarrho}{\mathord}{cmletters}{"25}    \let\varrho\myvarrho
\DeclareMathSymbol{\myvarsigma}{\mathord}{cmletters}{"26}  \let\varsigma\myvarsigma
\DeclareMathSymbol{\myvarphi}{\mathord}{cmletters}{"27}    \let\varphi\myvarphi

\DeclareMathOperator{\Hom}{Hom}

\DeclareMathOperator{\Spec}{Spec}
\DeclareMathOperator{\Trop}{Trop}

\DeclareMathOperator{\supp}{supp\,}

\DeclareMathOperator{\Alg}{{Alg}}
\DeclareMathOperator{\Rings}{{Rings}}
\DeclareMathOperator{\Bands}{{Bands}}

\newcommand\C{{\mathbb C}}
\newcommand\D{{\mathbb D}}

\newcommand\F{{\mathbb F}}
\newcommand\G{{\mathbb G}}
\renewcommand\H{{\mathbb H}}

\newcommand\K{{\mathbb K}}

\newcommand\M{{\mathbb M}}
\newcommand\N{{\mathbb N}}

\newcommand\R{{\mathbb R}}

\newcommand\T{{\mathbb T}}
\newcommand\U{{\mathbb U}}

\newcommand\cS{{\mathcal S}}

\newcommand\cX{{\mathcal X}}

\newcommand\Funpm{{\F_1^\pm}}

\renewcommand\char{\textup{char}\; }
\renewcommand{\min}{\textup{min}}
\renewcommand\geq{\geqslant}
\renewcommand\leq{\leqslant}

\newcommand{\gen}[1]{\langle #1 \rangle}
\newcommand{\genn}[1]{\langle\!\!\langle #1 \rangle\!\!\rangle}
\newcommand{\Genn}[1]{\Big\langle\!\!\Big\langle #1 \Big\rangle\!\!\Big\rangle}

\newcommand{\past}[2]{#1\!\sslash\!#2}
\newcommand{\pastgen}[3]{#1( #2 ) \!\sslash\!\langle\!\!\langle #3 \rangle\!\!\rangle}

\newcommand{\bandgen}[3]{#1[ #2 ] \!\sslash\!\langle #3 \rangle}
\renewcommand{\setminus}{\backslash}

\renewcommand\emptyset\varnothing

\title{On a theorem of Lafforgue}

\author{Matthew Baker}
\address{\rm Matthew Baker, School of Mathematics, Georgia Institute of Technology, Atlanta, USA}
\email{mbaker@math.gatech.edu}

\author{Oliver Lorscheid}
\address{\rm Oliver Lorscheid, University of Groningen, the Netherlands, and IMPA, Rio de Janeiro, Brazil}
\email{oliver@impa.br}

\hypersetup{pdftitle={On a theorem of Lafforgue},pdfauthor={Matthew Baker and Oliver Lorscheid}}

\begin{document}

\begin{abstract}
We give a new proof, along with some generalizations, of a folklore theorem (attributed to Laurent Lafforgue) that a rigid matroid (i.e., a matroid with indecomposable basis polytope) has only finitely many projective equivalence classes of representations over any given field.
\end{abstract}

\thanks{The authors thank David Speyer, Alex Fink, and Rudi Pendavingh for helpful discussions. We also thank BIRS for their hospitality hosting the workshop 23w5149: "Algebraic Aspects of Matroid Theory", during which the arguments in this paper emerged. Thanks also to Justin Chen, Eric Katz, and Bernd Sturmfels for their feedback on an earlier version of the paper. The first author was supported by NSF grant DMS-2154224 and a Simons Fellowship in Mathematics. The second author was supported by Marie Sk{\l}odowska Curie Fellowship MSCA-IF-101022339.}

\maketitle



\section{Introduction}
\label{introduction}

A matroid $M$ is called {\em rigid} if its base polytope $P_M$ has no non-trivial regular matroid polytope subdivisions.
Such matroids are interesting for a number of reasons; for example, a theorem of Bollen--Draisma--Pendavingh \cite{BDP18} 
asserts that for each prime number $p$, a rigid matroid is algebraically representable in characteristic $p$ if and only if it is linearly representable in characteristic $p$.
A folklore theorem, attributed to L. Lafforgue, asserts that a rigid matroid has at most finitely many representations over any field, up to projective equivalence.
This is mentioned without proof in a few places throughout the literature, for example in Alex Fink's Ph.D. thesis \cite[p.~10]{Fink10},
where he writes:

\begin{quote}
Matroid subdivisions have made prominent appearances in algebraic geometry. [\ldots] 
Lafforgue’s work implies, for instance, that a matroid
whose polytope has no subdivisions is representable in at most finitely many ways, up to the
actions of the obvious groups.
\end{quote}

We have been unable to find a proof of this result in the papers of Lafforgue cited by Fink \cite{Lafforgue99,Lafforgue03}, though a proof sketch appears in 
\cite[Theorem 7.8]{Fink15}.
In this paper, we provide a rigorous and efficient proof of Lafforgue's theorem, along with some new generalizations.

What is arguably most interesting about our approach to Lafforgue's theorem is that we deduce it from a purely algebraic statement which has nothing to do with matroids.
The only input from matroid theory needed is the fact that the {\em rescaling class functor} $\cX_M$ from pastures to sets is representable (see \autoref{sec:reformulation} below for further details).
We believe this to be a nice illustration of the power, and elegance, of the algebraic theory developed by the authors in \cite{Baker-Lorscheid20} and \cite{Baker-Lorscheid21b}.

\section{Reformulation and generalizations of Lafforgue's theorem} \label{sec:reformulation}

It is well-known to experts that a matroid $M$ is rigid if and only if every valuated matroid $\M$ whose underlying matroid is $M$ is rescaling equivalent to the trivially valuated matroid. Since we could not find a reference for this result, we provide a proof in \autoref{Appendix:Matroid Polytope Subdivisions}.

Recall from \cite{Baker-Lorscheid20} (see also \autoref{Appendix:Bands and Pastures}) that there is a category of algebraic objects called {\em pastures}, which generalize not only fields but also partial fields and hyperfields. 
According to \cite{Baker-Bowler19},
there is a robust notion of (weak) matroids over a pasture\footnote{Technically speaking, \cite{Baker-Bowler19} deals with {\em tracts}, not pastures, but the difference between the two is immaterial when considering weak matroids. For the sake of brevity, we do not define tracts in this paper, nor do we consider idylls or ordered blueprints (both of which play a prominent role in \cite{Baker-Lorscheid21b}).} $P$ such that (to mention just a few examples):
\begin{itemize}
\item Matroids over the Krasner hyperfield $\K$ are the same thing as matroids in the usual sense.
\item Matroids over the tropical hyperfield $\T$ are the same thing as valuated matroids.
\item Matroids over a field $K$ are the same thing as $K$-representable matroids, together with a choice of a matrix representation (up to the equivalence relation where two matrices are equivalent if they have the same row space).
\end{itemize}

For every matroid $M$ there is a functor $\cX_M$ from pastures to sets taking a pasture $P$ to the set of rescaling equivalence classes of (weak) $P$-representations of $M$.
A matroid $M$ is rigid if and only if $\cX_M(\T)$ consists of a single point. 
For a field $K$, $\cX_M(K)$ coincides with the set of projective equivalence classes of representations of $M$ over $K$.
Thus Lafforgue's theorem is equivalent to the assertion that if $\cX_M(\T)$ is a singleton, then $\cX_M(K)$ is finite for every field $K$. 

Recall from \cite{Baker-Lorscheid20} that for every matroid $M$, the functor $\cX_M$ is representable by a pasture $F_M$ canonically associated to $M$, called the {\em foundation} of $M$. Concretely, this means that $\Hom(F_M,P) = \cX_M(P)$ for every pasture $P$, functorially in $P$.

From this point of view, Lafforgue's theorem is equivalent to the assertion that if $\Hom(F_M,\T) = \{ 0 \}$, then $\Hom(F_M,K)$ is finite for every field $K$. This is the statement of Lafforgue's theorem that we actually prove in this paper. 
The advantage of this formulation is that it turns out to be a special case of a result which can be formulated purely in the language of pastures, without any mention of matroids! In fact, the algebraic incarnation of this result holds more generally with pastures (which generalize fields) replaced by {\em bands} (which generalize rings). 

See \autoref{Appendix:Bands and Pastures} for an overview of bands, including a definition, some examples, and the key facts needed for the present paper.

\subsection{An algebraic generalization of Lafforgue's theorem}

In order to state the algebraic result about bands which implies Lafforgue's theorem, we mention (see \autoref{prop: associated K-algebra} below)
that given a band $B$ and a field $K$, there is a canonically associated $K$-algebra $\rho_K(B)$ with the universal property that $\Hom_{\rm Band}(B,S) = \Hom_{K-{\rm alg}}(\rho_K(B),S)$ for every $K$-algebra $S$.
Moreover, if $B$ is finitely generated (which is the case, for example, when $B=F_M$ for some matroid $M$), then so is $\rho_K(B)$.

If $B$ is finitely presented, the set $\Hom(B,\T)$ has the structure of a finite polyhedral complex $\Sigma_B$; cf.\ \autoref{rem:dimension}.
Moreover, if $K$ is a field, the set $\Hom_{\rm Band}(B,K)$ is equal to $\Hom_{K-{\rm alg}}(\rho_K(B),K)$, which is in turn equal to the set $X_{B,K}(K)$ of $K$-points of the finite type affine $K$-scheme $X_{B,K} := \Spec(\rho_K(B))$.
(When $B=F_M$ for a matroid $M$, we call $X_{B,K}$ the {\em reduced realization space} of $M$ over $K$.)

Our first generalization of Lafforgue's theorem is as follows:

\begin{thm}
\label{thm:GeneralizedLafforgue1}
For every finitely generated band $B$ and every field $K$, we have the inequality $\dim X_{B,K} \leq \dim \Sigma_B$.
In particular, if $\Hom(B,\T) = \{ 0 \}$, then $\dim \Sigma_B = 0$ and thus $X_{B,K}(K) = \Hom(B,K)$ is finite for every field $K$.
\end{thm}

Applying \autoref{thm:GeneralizedLafforgue1} to $B=F_M$ immediately gives:

\begin{cor}[Lafforgue]
\label{cor:Lafforgue}
If $M$ is a rigid matroid, then $\cX_M(K)$ is finite for every field $K$. 
\end{cor}

In the terminology of \autoref{rem:LocalDressian}, \autoref{thm:GeneralizedLafforgue1} in the case $B=F_M$ says precisely that for any field $K$, the dimension of the reduced realization space of $M$ over $K$ is bounded above by the dimension of the reduced Dressian of $M$.

\subsection{A relative version of Lafforgue's theorem}

Rudi Pendavingh (private communication) asked if there might be a relative version of Lafforgue's theorem with respect to minors of $M$. 
More precisely, Pendavingh asked the following question: Suppose $N$ is an (embedded) minor of $M$ with the property that a valuated matroid structure on $M$ is determined, up to rescaling equivalence, by its restriction to $N$.
Is it then true that, for every field $K$, there are (up to projective equivalence) at most finitely many extensions of each $K$-representation of $N$ to a $K$-representation of $M$?

We answer Pendavingh's question in the affirmative, proving the following algebraic generalization of \autoref{cor:Lafforgue}:

\begin{thm}
\label{thm:GeneralizedLafforgue2}
Let $K$ be an algebraically closed valued field, and let $v : K \to \T$ be a non-trivial valuation.
If $f : B_1 \to B_2$ is a homomorphism of finitely generated bands, then the fiber dimension of $f_K : \Hom(B_2,K) \to \Hom(B_1,K)$ is bounded above by the fiber dimension of $f_{\T} : \Hom(B_2,\T) \to \Hom(B_1,\T)$, i.e.,
if $x \in \Hom(B_1,K)$ and $x'$ is the image of $x$ in $\Hom(B_1,\T)$, then $\dim f_K^{-1} (x) \leq \dim f_{\T}^{-1} (x')$.

In particular, setting $B_1 = F_N$ and $B_2 = F_M$ when $N$ is an embedded minor of a matroid $M$, we find that if the induced map $\cX_N(\T) \to \cX_M(\T)$ has finite fibers
(i.e., a valuated matroid structure on $N$ has at most finitely many extensions to $M$, up to rescaling equivalence)
then, for every field $k$, the natural map $\cX_N(k) \to \cX_M(k)$ has finite fibers, i.e., every $k$-representation of $N$ has at most finitely many extensions to $M$, up to projective equivalence.
\end{thm}

Note that Lafforgue's theorem (\autoref{cor:Lafforgue}) follows from the special case of \autoref{thm:GeneralizedLafforgue2} where $N$ is the trivial (empty) matroid and $f_{\T} : \Hom(B_2,\T) \to \Hom(B_1,\T)$ has finite fibers.

\section{Some examples}

In this section we present examples of both rigid and non-rigid matroids (see \autoref{Appendix:Bands and Pastures} for some details on our notation).

\begin{ex}[Dress--Wenzel]
In \cite[Theorem 5.11]{Dress-Wenzel92b}, Dress and Wenzel showed that if the inner Tutte group $F_M^\times$ of the matroid $M$ is finite, then $M$ is rigid. From our point of view, this is clear, since the inner Tutte group is the multiplicative group of the foundation (cf.~\cite[Corollary 7.13]{Baker-Lorscheid21b})
and a non-trivial homomorphism $F_M \to {\mathbb T}$ of pastures would give, in particular, a nonzero group homomorphism $F_M^\times \to ({\mathbb R},+)$; however the only torsion element of $({\mathbb R},+)$ is 0. 

For example:
\begin{enumerate}
\item The foundation of the Fano matroid $F_7$ is ${\mathbb F}_2$, so $F_7$ is rigid. More generally, any binary matroid has foundation equal to either $\F_1^\pm$ or $\F_2$ \cite[Corollary 7.32]{Baker-Lorscheid21b} and so it is rigid.
\item The foundation of the ternary spike $T_8$ is $\F_3$ (see \cite[Proposition 8.9]{Part2}), so $T_8$ is also rigid.
\item Dress and Wenzel prove in \cite[Corollary 3.8]{Dress-Wenzel92b} that the inner Tutte group of any finite projective space of dimension at least 2 is finite, which provides a wealth of additional examples of rigid matroids. 
\item Since the automorphism group of the ternary affine plane $M = {\rm AG}(2,3)$ acts transitively, all single-element deletions are isomorphic to each other. 
Let $M'$ be any of these deletions. By \cite[Proposition 6.2]{Part2}, the foundation of $M'$ is equal to the hexagonal (or sixth-root-of-unity) partial field  $\H \ = \ \pastgen\Funpm{T}{T^3 + 1, T-T^2 - 1}$, whose multiplicative group is the group of sixth roots of unity in $\C$. Therefore $M'$ is rigid.
\end{enumerate}
\end{ex}


It is not true that a matroid $M$ is rigid if and only if its inner Tutte group (or, equivalently, its foundation) is finite. For example:

\begin{ex}[suggested by Rudi Pendavingh]
Let $M$ be the Betsy Ross matroid (cf.~\cite[Figure 3.3]{vanZwam09}, where $M$ is also called $B_{11}$).
Using the Macaulay2 software described in \cite{Chen-Zhang}, we have checked that $F_M$ is the (infinite) golden ratio partial field $\G \ = \ \pastgen\Funpm{T}{T^2 - T - 1}$.
One checks easily that ${\rm Hom}({\mathbb G},{\mathbb T})$ is trivial, so $M$ is rigid; in particular, the converse of the statement ``$F_M$ finite implies $M$ rigid'' is not true.
It is also easy to see directly that ${\mathbb G}$ admits only finitely many homomorphisms to any field.
\end{ex}

\begin{ex}
The matroid $U_{2,4}$ is not rigid, since its foundation is the near-regular partial field $\U \ = \ \pastgen\Funpm{T_1,T_2}{T_1 + T_2 - 1}$, which admits infinitely many different homomorphisms to ${\mathbb T}$ (map $T_1$ to 1 and $T_2$ to any element less than or equal to 1, or vice-versa).
And for any field $K$, the reduced realization space $\cX_M(K)$ is equal to $K \setminus \{ 0,1 \}$, so in particular it is infinite whenever $K$ is.
The base polytope of $U_{2,4}$ is an octahedron, which admits a regular matroid decomposition into two tetrahedra (see \cite[p.~189]{Maclagan-Sturmfels15} for a nice visualization).
\end{ex}

\begin{ex}
 The non-Fano matroid $M=F_7^-$ is not rigid, and it provides an example for which the dimension of the reduced realization spaces $\cX_M(K)$ and $\cX_M(\T)$ jumps. The foundation of $M$ is the dyadic partial field $\D=\pastgen\Funpm{T}{T+T-1}$ by \cite[Prop.~8.4]{Part2}, and there is at most one homomorphism $F_M=\D\to K$ into any field $K$, sending $T$ to the multiplicative inverse of $2$ (if it exists, i.e.,\ if $\char K\neq 2$). In contrast, there are infinitely many homomorphisms $\D\to\T$ (parametrized by the image of $f(T)\in\T$). So $\dim\cX_M(K)=0<1=\dim\cX(\T)$.
\end{ex}

\section{Proof of the main theorems}

The key fact needed for the proof of \autoref{thm:GeneralizedLafforgue1} is the following theorem of Bieri and Groves \cite[Theorem A]{Bieri-Groves84}, which is a cornerstone of tropical geometry. For the statement, recall that a {\em semi-valuation} from a ring $R$ to $\overline\R=\R\cup\{+\infty\}$ is a map $v : R \to \overline\R$ such that $v(0)=+\infty$, $v(xy)=v(x)+v(y)$, and $v(x+y) \geq \min \{ v(x),v(y) \}$ for all $x,y \in R$.
(The map $v$ is called a {\em valuation} if, in addition, $v(x)=+\infty$ implies that $x=0$.) If $R$ is a $K$-algebra, where $K$ is a valued field (i.e., a field endowed with a valuation $v : K \to \overline\R$), a {\em $K$-semi-valuation} is a semi-valuation which restricts to the given valuation on $K$.

\begin{thm}[Bieri--Groves] \label{thm: Bieri-Groves}
Let $K$ be a field endowed with a real valuation $v$,
and suppose $R$ is a finitely generated $K$-algebra with Krull dimension equal to $n$, having generators $T_1,\ldots,T_n$. Let $X = {\rm Spec}(R)$ be the corresponding affine $K$-scheme.
 Then the set
\[
{\rm Trop}(X) := \{ (v(T_1),\ldots,v(T_n)) \; | \; v : R \to \overline\R \textrm{ is a $K$-semi-valuation} \}
\]
is a polyhedral complex of dimension $\dim({\rm Trop}(X)) = \dim X$.
\end{thm}

\begin{rem}
Bieri and Groves assume that $X$ is irreducible and show, more precisely, that ${\rm Trop}(X)$ has {\em pure} dimension $n$. Our formulation of the Bieri--Groves theorem (which does not include the purity statement) follows immediately from theirs by decomposing $X$ into irreducible components.
\end{rem}

\begin{rem} 
More or less by definition, a semi-valuation on a ring $R$ is precisely the same thing as a homomorphism from $R$ to $\T$ in the category of bands, and if $K$ is a valued field then a $K$-semi-valuation on $R$ is the same thing as a homomorphism from $R$ to $\T$ which restricts to the given homomorphism $v : K \to \T$ on $K$.
\end{rem}

Let $K$ be a field, and let $\Alg_K$ denote the category of $K$-algebras, i.e.\ ring extensions $R$ of $K$ together with $K$-linear ring homomorphisms. 
We write $\Hom_K(R,S)$ for the set of $K$-algebra homomorphisms between two $K$-algebras $R$ and $S$.
Given a band $B$, we define the \emph{associated $K$-algebra} as
\[
 \rho_K(B) \ = \ K[B] \, / \, \gen{N_B},
\]
where $K[B]$ is the monoid algebra over $K$ and the elements of the nullset $N_B$ are interpreted as elements of $K[B]$ (cf.~\autoref{df:band}). 
It comes with a band homomorphism $\alpha_B:B\to \rho_K(B)$, which maps $a$ to $[a]$.

The other main ingredient needed for the proof of \autoref{thm:GeneralizedLafforgue1} is the following technical but important result:

\begin{prop}\label{prop: associated K-algebra}
 Let $K$ a field, $B$ be a band and $R=\rho_K(B)$ the associated $K$-algebra.
 \begin{enumerate}
  \item\label{alg1} The homomorphism $\alpha_B:B\to \rho_K(B)$ is initial for all homomorphisms from $B$ to a $K$-algebra, i.e.,\ for every $K$-algebra $S$ the natural map
  \[
   \Hom_K(R,S) \ \stackrel{\alpha_B^\ast}{\longrightarrow} \ \Hom(B,S)
  \]
  is a bijection.
  \item\label{alg2} Assume we are given a valuation $v_K : K \to \T$, and that
 $B$ is finitely generated by $a_1,\dotsc,a_n$. Let $T_i=\alpha_B(a_i)$ for $i=1,\dotsc,n$, and let $X=\Spec R$. Let $\exp^n:\overline\R^n\to\T^n$ be the coordinate-wise exponential map. Then the $T_i$ generate $R$ as a $K$-algebra, and 
  \[
   \exp^n \Big(\Trop(X)\Big) \ \subset \ \Hom(B,\T)
  \]
  as subsets of $\T^n$. 
 \end{enumerate}
\end{prop}

\begin{proof}
 We begin with \eqref{alg1}. The map $\alpha_B^\ast$ is injective since $R$ is generated by the subset $\alpha_B(B)$, and therefore every homomorphism $f:R\to S$ is determined by the composition $f\circ\alpha_B:B\to S$. In order to show that $\alpha_B^\ast$  is surjective, consider a band homomorphism $f:B\to S$, which is, in particular a multiplicative map. Therefore it extends (uniquely) to a $K$-linear homomorphism $\hat{f}:K[B]\to S$ from the monoid algebra $K[B]$ to $S$. For every $\sum a_i\in  N_B$, we have $\sum f(a_i)\in N_S$ by the definition of a band homomorphism. By the definition of $N_S$, this means that $\sum f(a_i)=0$ in $S$. Thus $\hat{f}$ factorizes through $\bar f:R=K[B]/\gen{N_B}\to S$, and, by construction, we have $f=\bar f\circ\alpha_B=\alpha_B^\ast(\bar f)$. This establishes \eqref{alg1}.
 
 We continue with \eqref{alg2}. Since $B$ is generated by $a_1,\dotsc,a_n$ as a pointed monoid and $\alpha_B(B)$ generates $R$ as a $K$-algebra, $R$ is generated as a $K$-algebra by $T_1,\dotsc,T_n$. In order to verify that $\exp^n(\Trop(X))\subset\Hom(B,\T)$, consider a point $(v(T_1),\dotsc,v(T_n))\in \Trop(X)$, where $v:R\to\overline\R$ is a $K$-semi-valuation. Post-composing $v$ with $\exp$ yields a seminorm $v':R\to\T$, which is, equivalently, a band homomorphism. Pre-composing $v'$ with $\alpha_B$ yields a band homomorphism $v'':B\to\T$, which is an element of $\Hom(B,\T)$. By construction, $\exp^n(v(T_1),\dotsc,v(T_n))=v''$, which establishes the last assertion.
\end{proof}

\begin{rem}\label{rem:dimension}
 \begin{enumerate}
\item Under the assumptions of \autoref{prop: associated K-algebra}.\eqref{alg2}, $\Hom(B,\T)$ embeds as a subspace of $\T^n$, which has a well-defined (Lebesgue) covering dimension in the sense of \cite[Chapter 3]{Pears75}. 
As discussed in \cite{Lorscheid22}, the subspace topology of $\Hom(B,\T)\subset\T^n$ is equal to the compact-open topology for $\Hom(B,\T)$ with respect to the discrete topology for $B$ and the natural order topology for $\T$, which shows that the dimension of $\Hom(B,\T)$ does not depend on the embedding into $\T^n$.
\item With the topologies just described, $\exp^n$ defines a continuous injection from $\Trop(X)$ to $\Hom(B,\T)$ which identifies the former with a closed subspace of the latter. In particular, \cite[Prop. 3.1.5]{Pears75} shows that $\dim \Trop(X) \leq \dim \Hom(B,\T)$.
  \item If in addition to the assumptions of \eqref{alg2}, $N_B$ is finitely generated as an ideal of $B^+$, then $\Hom(B,\T)$ is a tropical pre-variety in $\T^n$ and is therefore the underlying set of a finite polyhedral complex. The dimension of $\Hom(B,\T)$ as a polyhedral complex is equal to its covering dimension \cite[Theorem 2.7 and Section~{3.7}]{Pears75}. 
  \end{enumerate}
\end{rem}

\begin{proof}[Proof of \autoref{thm:GeneralizedLafforgue1}]
 Let $v : K \to \T$ be a valuation (which we can take to be the trivial valuation if we like).
 Let $\alpha_B:B\to R$ be the canonical homomorphism to the associated $K$-algebra $R=\rho_K(B)$, cf.~\autoref{prop: associated K-algebra}. 
 Let $a_1,\dotsc,a_n\in B$ be a set of generators for $B$, and for $i=1,\dotsc,n$ let $T_i=\alpha_B(a_i)$. By \autoref{prop: associated K-algebra}, the $T_i$
 generate $R$ as a $K$-algebra, i.e.,\ $R=K[T_1,\dotsc,T_n]/I$ for some ideal $I$. 
 
 Let $X=\Spec R$, so that $X(K) = \Hom_K(R,K)$. \autoref{prop: associated K-algebra} yields a commutative diagram
 \[
 \begin{tikzcd}
  X(K) = \Hom_K(R,K) \arrow[r, "\simeq"] \arrow[d] & \Hom(B,K) \ar[d] \\
  \Trop(X) \ar[right hook->,r,"\exp^n"] & \Hom(B,\T) 
 \end{tikzcd}
 \]
 where the right-hand vertical map is obtained by composing with $v : K \to \T$ and the 
 left-hand vertical map is induced by composing the embedding of $X(K) = \Hom_K(R,K)$ into $K^n$ via $\phi \mapsto (\phi(T_i))_{i=1}^n$ with the coordinate-wise absolute value $v_K^n:K^n\to\T^n$. 
 
By the Bieri--Groves theorem (\autoref{thm: Bieri-Groves}), the dimension of the affine variety $X$ is equal to the dimension of $\Trop(X)$, as defined in \autoref{rem:dimension}. 
Using \autoref{prop: associated K-algebra}\eqref{alg2} and \autoref{rem:dimension}\eqref{alg2}, we conclude that
 \[
  \dim\Big(X \Big) \ = \ \dim\Big(\Trop(X)\Big) \ \leq \ \dim\Big(\Hom(B,\T)\Big),
 \]
 as desired.
\end{proof}

\begin{proof}[Proof of \autoref{thm:GeneralizedLafforgue2}]
Suppose $f : B_1 \to B_2$ is a band homomorphism. Choose generators $x_1,\ldots,x_m$ for $B_1$. Completing $f(x_1),\ldots,f(x_m)$ to a set of generators for $B_2$ if necessary, we find a generating set $y_1,\ldots,y_n$ for $B_2$ with $m \leq n$ such that $f(x_i)=y_i$ for $i=1,\ldots,m$.
Setting $X = \Spec(\rho_K(B_1))$ and $Y = \Spec(\rho_K(B_2))$, and letting $\Trop(X)$ (resp. $\Trop(Y)$) be the tropicalization of $X$ with respect to $\alpha_{B_1}(x_1),\ldots,\alpha_{B_1}(x_m)$ (resp. $\alpha_{B_2}(y_1),\ldots,\alpha_{B_2}(y_n)$),
we obtain a commutative diagram
\[
 \begin{tikzcd}
  Y(K) \arrow{r}{f_K}\arrow[d]
  & X(K) \arrow[d] \\
   \Trop(Y) \arrow{r}{f_{\T}} \arrow[d,hook]
  & \Trop(X) \arrow[d,hook] \\
  \Hom(B_2,\T) \arrow[r]
  & \Hom(B_1,\T)
 \end{tikzcd}
\]
Since $\Trop(Y)$ is a closed subspace of $\Hom(B_2,\T)$ (resp. $\Trop(X)$ is a closed subspace of $\Hom(B_1,\T)$), 
it suffices to prove that if $x \in X(K)$ and $x' = \Trop(x) \in \Trop(X)$, then $\dim f_K^{-1} (x) \leq \dim f_{\T}^{-1} (x')$.


To see this, 
write $f_K^{-1}(x) = Z(K)$ with $Z$ an affine subscheme of $Y$.
If we pull back the functions $\alpha_{B_2}(y_1),\ldots,\alpha_{B_2}(y_n)$ to a set of generators for the affine coordinate ring of $Z$, we obtain a commutative diagram
\[
 \begin{tikzcd}
  Z(K) \arrow[r,hook] \arrow[d] & Y(K) \arrow[r] \arrow[d]
  & X(K) \arrow[d] \\
  \Trop(Z) \arrow[r,hook] & \Trop(Y) \arrow[r] 
  & \Trop(X)
 \end{tikzcd}
\]

Applying the Bieri--Groves theorem to $Z$, 
we find that the image of $Z(K)$ under ${\rm Trop}$ has dimension equal to $\dim f_{K}^{-1} (x)$.
In addition, the natural map $\Trop(Z) \to \Trop(Y)$ identifies $\Trop(Z)$ with a closed subspace of $\Trop(Y)$, since $\Trop(Z)$ (resp. $\Trop(Y)$) is the topological closure of $Z(K)$ (resp. $Y(K)$) in $\T^n$ (cf.~\cite[Proposition 2.2]{Payne09}).
By construction, $\Trop(Z)$ is in fact a closed subspace of $\dim f_{\T}^{-1} (x')$.
This means that $\dim f_K^{-1} (x) = \dim \Trop(Z) \leq \dim f_{\T}^{-1} (x')$ as desired.
\end{proof}

\appendix

\section{Pastures and Bands} \label{Appendix:Bands and Pastures}

More details pertaining to the following overview of bands and pastures can be found in \cite{Baker-Jin-Lorscheid}.

In this text, a \emph{pointed monoid} is a (multiplicatively written) commutative semigroup $A$ with identity $1$, together with a distinguished element $0$ that satisfies $0\cdot a=0$ for all $a\in A$. The \emph{ambient semiring of $A$} is the semiring $A^+=\N[A]/\gen 0$, which consists of all finite formal sums $\sum a_i$ of nonzero elements $a_i\in A$. Note that $A$ is embedded as a submonoid in $A^+$, where $0$ is identified with the empty sum. An \emph{ideal of $A^+$} is a subset $I$ that contains $0$ and is closed under both addition and multiplication by elements of $A^+$.

\begin{df} \label{df:band}
 A \emph{band} is a pointed monoid $B$ together with an ideal $N_B$ of $B^+$ (called the \emph{nullset}) such that for every $a\in A$, there is a unique $b\in A$ with $a+b\in N_B$. We call this $b$ the \emph{additive inverse of $a$}, and we denote it by $-a$. A \emph{band homomorphism} is a multiplicative map $f:B\to C$ preserving $0$ and $1$ such that $\sum a_i\in N_B$ implies $\sum f(a_i)\in N_C$. This defines the category $\Bands$.
\end{df} (

For a subset $S$ of $B^+$, we denote by $\genn{S}$ the smallest ideal of $B^+$ that contains $S$ and is closed under the \emph{fusion axiom} (cf.~\cite{Baker-Zhang23})
\begin{enumerate}[label=\rm(F)]
 \item if $c+\sum a_i$ and $-c+\sum b_j$ are in $\genn{S}$, then $\sum a_i+\sum b_j$ is in $\genn{S}$.
\end{enumerate}

\begin{df}
 A band $B$ is \emph{finitely generated} if it is finitely generated as a monoid. It is a \emph{finitely presented fusion band}, which we abbreviate by simply saying that $B$ is \emph{finitely presented}, if it is finitely generated and $N_B=\genn{S}$ for a finite subset $S$ of $N_B$.
\end{df}

The \emph{unit group of $B$} is the submonoid $B^\times=\{a\in B\mid ab=1\text{ for some }b\in B\}$ of $B$, which is indeed a group. 

\begin{df}
 A \emph{pasture} is a band $P$ with $P^\times=P-\{0\}$ and 
 \[
  N_P \ = \ \Genn{ \, a+b+c \in P^+ \; \Big| \; a+b+c\in N_P \, }.
 \]
\end{df}

\begin{ex}\label{ex: bands}
 Every ring $R$ is a band, with nullset $N_R=\{\sum a_i\mid \sum a_i = 0 \text{ in }R\}$. In fact, this defines a fully faithful embedding $\Rings\to\Bands$. Every field is a pasture.
 
 The following examples of interest are bands which are not rings (we write $a-b$ for $a+(-b)$):
 \begin{itemize}
  \item The \emph{regular partial field} is the pasture $\Funpm=\{0,1,-1\}$ with nullset
  \[
   N_\Funpm \ = \ \Big\{n.1+n.(-1) \, \Big| \, n\geq0 \Big\} \ = \ \genn{1-1}.
  \]

  \item The \emph{Krasner hyperfield} is the pasture $\K=\{0,1\}$ with nullset 
  \[
   N_\K \ = \ \N-\{1\} \ = \ \genn{1+1,\ 1+1+1}.
  \]
  \item The \emph{tropical hyperfield} is the pasture $\T=\R_{\geq0}$ with nullset
  \[\textstyle
   N_\T \ = \ \{0\} \ \cup \ \Big\{ \sum a_i \; \Big| \; a_1,\dotsc,a_n\text{ assumes its maximum at least twice}\Big\}.
  \]
 \end{itemize}
 Examples of band homomorphisms are the inclusion $\K\hookrightarrow \T$ and the surjection $\T\to\K$ that sends every nonzero element to $1$. A band homomorphism $R\to\T$ from a ring $R$ into $\T$ is the same thing as a non-archimedean seminorm. In particular, the trivial absolute value on a field $K$ is the unique band homomorphism $K\to\T$ that factors through $\K$.
\end{ex}

The pasture $\Funpm$ is an initial object in $\Bands$, i.e.,\ every band $B$ comes with a unique homomorphism $\Funpm\to B$. This leads to a description $B=\bandgen\Funpm{T_i\mid i\in I}{S}$ of $B$ in terms of generators $\{T_i\mid i\in I\}$ and relations $S\subset B^+$, in the sense that $\{T_i\}\cup\{0,-1\}$ generates $B$ as a monoid, $S$ generates the ideal $N_B$, and 
$S$ contains a complete set of binary relations between the signed products $x=\pm T_{i_1}\dotsb T_{i_r}$ of the $T_i$, i.e.,\ if $x-y\in S$ then $x=y$ as elements of $B$.

Similarly, we write $P=\pastgen\Funpm{T_i\mid i\in I}{S}$ for a pasture $P$ if $P^\times$ is generated as a group by $\{T_i\mid i\in I\}$ and $-1$, if $N_P=\genn{S}$, and if $S$ contains a complete set of binary relations between the signed products of the $T_i$. For example,
\[
 \K \ = \ \past\Funpm{\genn{1+1,\ 1+1+1}}, \quad \text{and} \quad \F_5 \ = \ \pastgen\Funpm{T}{T^2+1,\ T-1-1}.
\]

\section{Valuated matroids and subdivisions of the basis polytope} \label{Appendix:Matroid Polytope Subdivisions}

In this section, we show that a matroid is rigid if and only if it has a unique rescaling class over $\T$. We begin with some observations and recall some results from the literature.

For a pasture $F$, we can identify isomorphism classes of a (weak) Grassmann-Pl\"ucker function $\Delta$ with the corresponding {\em Pl\"ucker vector} $(\Delta(I))_{I \in \binom Er} \in {\mathbb P}^{\binom Er}(F)$.
We call this Pl\"ucker vector a {\em representation} of $M$, and by abuse of terminology we use the terms ``Grassmann-Pl\"ucker function'' and ``Pl\"ucker vector'' interchangeably.

Every matroid $M$ can be (uniquely) represented over $\K$ by the Grassmann-Pl\"ucker function $\Delta_M:\binom Er\to\K$ which sends an $r$-subset $I$ of $E$ to $1$ if it a basis of $M$ and to $0$ otherwise. 
Post-composing $\Delta_M$ with the inclusion $\K\hookrightarrow \T$ defines the \emph{trivial representation of $M$}, which shows that $M$ has at least one rescaling class over $\T$.

Recall that the {\em basis polytope} $P_M$ of $M$ is the convex hull of the points $e_I=\sum_{i\in I} e_i\in \R^n$ for which $I$ is a basis of $M$. Let $\Delta:\binom Er\to\T$ be a Pl\"ucker vector for $M$, i.e.,\ $\supp(\Delta)=\supp(\Delta_M)$. 

Let $\cS_\Delta=\{e_I\mid \Delta(I)\neq0\}$ be the support of $\Delta$, considered as a subset of $\R^n$. Post-composing with $\log$ yields a function $\tilde\Delta:\cS_\Delta\to\R$ whose graph $\Gamma$ is a subset of $\R^n\times\R$. The convex closure of $\Gamma$ has a unique coarsest structure as a polyhedral complex. The lower faces of this polyhedral complex are those faces for which the last coordinate of the outward 
normal vector is negative. Omitting this last coordinate projects these faces onto $P_M$ and defines a polyhedral subdivision of $P_M$ called the {\em regular subdivision associated to $\Delta$}, see e.g. \cite[Definition 2.3.8]{Maclagan-Sturmfels15}.

By a theorem of Speyer (cf.~\cite[Prop.~2.2]{Speyer08}), this subdivision of $P_M$ is a {\em matroid subdivision}, i.e., all faces of the subdivision are themselves matroid polytopes, and conversely every regular matroid subdivision of $P_M$ comes from a $\T$-representation of $M$ (see also \cite[Lemma 4.4.6]{Maclagan-Sturmfels15} and \cite[Thm.~10.35]{Joswig21}).

\begin{prop}\label{prop: rigid matroids}
 A matroid $M$ is rigid if and only if $M$ has a unique rescaling class over $\T$.
\end{prop}

\begin{proof}
 Let $r$ be the rank and $E=\{1,\dotsc,n\}$ the ground set of $M$. Let $\Delta:\binom Er\to\T$ be a tropical Pl\"ucker vector for $M$, and let $\cS_\Delta$ be as above. 
 
 By definition, $M$ is rigid if and only if $P_M$ admits only the trivial regular matroid subdivision. Since none of the points of $\cS_\Delta$ lies in the convex closure of the other points, $\Delta:\binom Er\to\T$ induces the trivial matroid subdivision if and only if the subset $\Big\{(e_I,\tilde\Delta(I))\mid I\in\cS_\Delta\Big\}$ of $\R^n\times\R$ is contained in an affine hyperplane $H$. 
 
 In this case, let $x_ie_i$ be the unique intersection point of $H$ with the coordinate axis generated by $e_i$ (in the case of a loop $i$ of $M$ there is no such intersection point, and we can formally put $x_i=+\infty$). Then $\tilde\Delta(I)=\sum_{k=1}^r x_{i_k}e_{i_k}$ for $I\in\cS_\Delta$. Rescaling $\Delta$ by $t=(\exp(-x_i)\mid i=1,\dotsc,n)$ yields a Pl\"ucker vector $\Delta_0=t.\Delta:\binom Er\to\T$ for which 
 \[
  \tilde\Delta_0(I) \ = \ \tilde\Delta(I)-\sum_{k=1}^r x_{i_k}e_{i_k} \ = \ 0
 \]
 for every $I\in\cS_\Delta$. Thus $\tilde\Delta_0$ is the trivial representation of $M$.
 Conversely, rescaling $\Delta_0$ yields a Pl\"ucker vector $\Delta$ for which $\Big\{(e_I,\tilde\Delta(I))\mid I\in\cS_\Delta\Big\}$ is contained in an affine hyperplane, which concludes the proof.
\end{proof}

\begin{rem} \label{rem:LocalDressian}
The (local) {\em Dressian} of a matroid $M$ (cf.~\cite{OPS19}) is a polyhedral complex $\Delta_M$ whose underlying set consists of all $\T$-representations of $M$; the polyhedral structure is defined by the 3-term tropical Pl\"ucker relations.
One can show using \cite[Cor. 18]{OPS19} that the lineality space of $\Delta_M$ is precisely the set of valuations on $M$ which are projectively equivalent to the trivial valuation.
The topological space ${\rm Hom}(F_M,{\mathbb T})$ considered in the body of this paper can then be naturally identified with $\Delta_M$ modulo its lineality space, which we call the {\em reduced Dressian} $\overline{\Delta}_M$.
(We omit the details, as it would take us too far afield into a somewhat lengthy discussion of various topologies and polyhedral structures.) 
See \cite[Section 3]{Brandt-Speyer22} for an algorithm for computing the Dressian and/or reduced Dressian of a matroid $M$, and also (in Section 5) some interesting counterexamples to plausible-sounding assertions.
\end{rem}


\begin{small}
 \bibliographystyle{plain}
 \bibliography{matroid}
\end{small}

\end{document}